\documentclass{amsart}

\newtheorem{thm}{Theorem}[section]

\newtheorem*{main}{Main Theorem}
\newtheorem{lemma}[thm]{Lemma}

\theoremstyle{definition}

\newtheorem{remark}{Remark}[section]

\newtheorem{cor}[thm]{Corollary}

\newtheorem{pro}[thm]{Proposition}

\usepackage{amssymb}

\title[Measures of Intermediate Entropies]{Measures of Intermediate Entropies for Skew Product Diffeomorphisms}
\author[Peng Sun]{}


 \email{peng\_sun@math.psu.edu}

\begin{document}

\maketitle

\centerline{\scshape Peng Sun }
\medskip
{\footnotesize
 \centerline{Department of mathematics}
   \centerline{The Pennsylvania State University}
   \centerline{University
Park, PA 16802, USA}
} 

\begin{abstract}
   In this paper we study a skew product map $F$ with a measure $\mu$
   of positive entropy.
We show that if on the fibers the map are $C^{1+\alpha}$ diffeomorphisms with
nonzero Lyapunov exponents, then $F$ has ergodic measures of arbitrary intermediate entropies. To construct these measures
we find a set on which the return map is a skew product with horseshoes
along fibers. We can control the average return time and show the maximum
entropy of these measures can be arbitrarily close to $h_\mu(F)$.
\end{abstract}

\section{Introduction}
Entropy has been one of the centerpieces in dynamics. It reflects
the complexity of the system. For smooth systems, positive topological entropy
comes from some (partially) hyperbolic structure, and is conjectured to be
accompanied with plenty of invariant measures. The
work presented here explores answers to such conjectures in a case of
skew product maps.

For a compact Riemannian manifold $M$ and a $\mathrm{C}^{1+\alpha}$
($\alpha>0$, this means the derivatives are $\alpha$-H\"older continuous)
 diffeomorphism $f$ on $M$, if $f$ preserves an ergodic measure $\mu$, then
 there are real numbers $\lambda_k$, $k=1, 2, \dots, l$,  such that for $\mu$-almost every point $x$, there are subspaces $E_k(x)$ of the tangent space $T_xM$
such that for any vector $v\in E_k(x)\backslash\{0\}$,
$$\lambda(v):=\lim_{n\to\infty}\frac{\log\|df^{n}v\|}n=\lambda_k$$
Here $l$ is some integer number no more than the dimension of $M$ and these subspaces
are invariant of $df$, the derivative of $f$. These numbers $\lambda_k$ are
called Lyapunov exponents. We say $\mu$ is a hyperbolic measure if all Lyapunov
exponents are nonzero. Three decades ago,
A. Katok established a well-known result as following:

\begin{thm}\label{ThKa}\textnormal{(Katok, \cite{Ka2, KM1})}
If the metric entropy $h_\mu(f)>0$ and $\mu$ is an ergodic hyperbolic measure, then
for any $\epsilon>0$, there is a hyperbolic horseshoe $\Lambda\subset M$ such that
$h(f|_\Lambda)>h_\mu(f)-\epsilon.$
\end{thm}

This theorem has an interesting corollary: Under the conditions of the theorem,
there are ergodic measures $\mu_\beta$ such that $h_{\mu_\beta}(f)=\beta$
for any real number $\beta\in[0, h_\mu(f)]$, i.e. all possible entropies of
ergodic measures form an interval. Since the horseshoe map is a
full shift, these measures can be constructed
by taking sub-shifts or properly assigning weights to different symbols.
Existence of these measures of intermediate entropies exhibits the complicated
structure of the system.

Till now it is not known whether every smooth system ($C^{1+\alpha}$ diffeomorphism)
on a compact manifold with positive (topological) entropy has this property. In general, such a system may not
have a hyperbolic measure. Herman constructed a well-known example, which
is a minimal $C^\infty$ diffeomorphism with positive topological entropy \cite{He1}. So in the general case
even a closed invariant subset should not be expected. Fortunately, minimality
may not prevent the system from having measures of intermediate entropies,
which is the case in Herman's example. Herman suggested the question whether
for every smooth system positive topological entropy violates unique ergodicity. Katok conjectured
more ambitiously: these systems must have measures of arbitrary intermediate entropies.
 We are working towards this goal.

In this paper we deal with skew product maps with nonzero Lyapunov exponents
along fibers. Let $F=(g,f_x)$  be a skew product map
on the space $X\times Y$ preserving an ergodic measure $\mu=\int\sigma_x d\nu$.
We assume that $g$ is an invertible (mod 0) measure preserving transformation on the probability
space $(X,\nu)$. Then $g$ is ergodic.
For every $x\in X$, $f_x$
is a $C^{1+\alpha}$ diffeomorphism on the compact Riemannian manifold $Y$, sending
$\sigma_x$ to $\sigma_{g(x)}$.

\begin{main}\label{ThMain}
Assume that $h_\mu(F)>0$ and $h_\nu(g)=0$. If for almost every $z=(x,y)\in
X\times Y$ and every $v\in T_y(\{x\}\times Y)\backslash\{0\}$, the Lyapunov
exponent $$\lambda(v)=\lim_{n\to\infty}\frac{\log\|df_{g^{n-1}(x)}\cdots
df_{g(x)}df_x v\|}{n}\neq
0,$$
then $F$ has ergodic invariant measures of arbitrary intermediate entropies.
\end{main}

\begin{remark}
Our proof also works when $h_\nu(g)>0$. In this case it concludes that $F$ has ergodic measures of arbitrary entropies between $h_\nu(g)$ and $h_\mu(F)$. \end{remark}
In addition, we
assume $g$ has no periodic point. Otherwise the problem is reduced to Theorem
\ref{ThKa}. We have shown in \cite{Sun} that under the conditions there are
measures of zero entropy. Some lemmas are generalized and adapted in this
paper.

{\flushleft\emph{Acknowledgements}.} The author would like to thank Anatole Katok for numerous discussions and encouragement.

\section{Entropy and Separated Sets on Fibers}

In this section we discuss the entropy of the skew product and  obtain an estimate on the cardinality of the
$(m,\epsilon)$-separated set on each fiber, which is analogous to the definition
of metric entropy by Katok \cite{Ka1}. 

Let $\eta$ be a measurable partition of the fiber $Y$ with finite entropy $H_x(\eta)<\infty$ for almost every $x\in X$,
where $H_x(\eta)=-\sum_{\mathcal{C}\in\eta}\sigma_x(\mathcal{C})\log\sigma_x(\mathcal{C})$.
Let us put
$$\eta_x^n=\bigvee_{k=0}^{n-1}f_x^{-1}f_{g(x)}^{-1}\cdots f_{g^{k-1}(x)}^{-1}\eta$$

\begin{thm}\label{ThAR}\textnormal{(Abramov and Rohlin, \cite{AR})}
For every $\eta$, put
$$h^g(f,\eta)=\lim_{n\to\infty}\frac1n\int_X H_x(\eta_x^n)d\nu$$
The limit exists and it is finite. Let
$$h^g(f)=\sup_\eta
h^g(f,\eta)$$
$h^g(f)$ is called the fiber entropy. We have
$$h_\mu(F)=h_\nu(g)+h^g(f)$$
\end{thm}

For the skew product map we consider, $g$ is ergodic and,
for almost every $x\in X$, $\sigma_x\circ f_x^{-1}=\sigma_{g(x)}$. we have for almost every
$x$,
$$h^g_x(f,\eta)=\lim_{n\to\infty}\frac1n H(\eta_x^n)=h^g(f,\eta)$$

The following  is a version of Shannon-McMillan-Breiman Theorem for skew
product maps.

\begin{thm}\textnormal{(Belinskaja, \cite{Be})}\label{ThKL}
Let $\mathcal{C}_x^n(y)$ be the element of $\eta_x^n$ containing $y$.
For almost every $(x,y)\in X\times Y$,
$$\lim_{n\to\infty}-\frac 1n\sigma_x(\mathcal{C}_x^n(y))=h^g_x(f,\eta)=h^g(f,\eta)$$
\end{thm}

Fix $d$ a Riemannian metric on $Y$. Let $d_{n}^F(z,z')$ be an increasing system of metrics defined for $z,z'$ on
the same fiber 
$\{x\}\times
Y$ by:
$$d_{n}^F(z,z')=\max_{0\le i\le n-1}d(F^i z, F^i z')$$
For $x\in X$ and $\delta>0$, on the fiber $\{x\}\times Y,$ let $\mathcal{N}_x^F(n,\epsilon,\delta)$ be the minimal number
of $\epsilon$-balls in the $d_n^F$-metric needed to cover a set of $\sigma_x$-measure
at least $1-\delta$, and let $\mathcal{S}_x^F(n,\epsilon,\delta)$ be the
 maximal number of $(d_n^F,\epsilon)$-separated points we can find inside every set
of $\sigma_x$-measure at least $1-\delta$.

We can follow exactly Katok's argument for \cite[Theorem 1.1]{Ka1} and obtain
the analogous result:

\begin{thm}\label{ThSe}
If $F$ is ergodic, then for almost every $x\in X$ and every $\delta>0$,
$$h^g(f)=\lim_{\epsilon\to 0}\liminf_{n\to\infty}\frac{\log\mathcal{N}_x^F(n,\epsilon,\delta)}{n}=
\lim_{\epsilon\to 0}\limsup_{n\to\infty}\frac{\log\mathcal{N}_x^F(n,\epsilon,\delta)}{n}$$
And
$$h^g(f)=\lim_{\epsilon\to 0}\liminf_{n\to\infty}\frac{\log\mathcal{S}_x^F(n,\epsilon,\delta)}{n}=
\lim_{\epsilon\to 0}\limsup_{n\to\infty}\frac{\log\mathcal{S}_x^F(n,\epsilon,\delta)}{n}$$
\end{thm}

\section{Recurrence}

In this section we discuss some properties related to nontrivial recurrence
of the map. 

Theorem \ref{ThInt} is a generalization of \cite[Proposition 3.2]{Sun}. It
says that, for a subset of positive measure, if the conditional measures are uniformly bounded from below, then on
each fiber we can find points that return relatively faster, such that the
return time is integrable. In this paper we still use the special version
for the first return.

 Theorem \ref{ThSet} is crucial in the proof of the main theorem. For a complicated or even randomly selected return map on some set of positive measure,
if the return time is integrable, then this return map is one-to-one and
measure preserving on a smaller subset of positive measure.
We call this subset the \emph{kernel} of the return map. Moreover, we have an estimate of the integral of the return time, which may be used to estimate the size (measure) of the
kernel.

\begin{thm}\label{ThInt}\textnormal{(Integrability of Return Time)}
Let $P\subset X\times Y$ be a measurable subset. $B=\pi(P)\subset X$ is the
projection of $P$ on the base. For $x\in B$, denote $P\cap (\{x\}\times Y)$
by $P(x)$. Assume that $\nu(B)=\nu_{0}>0$ and there is $\sigma_0>0$ such that for
(almost) every $x\in B$,
$\sigma_x(P(x))>\sigma_0$. Hence $\mu(P)=\mu_0>\nu_0\sigma_0>0$. For (almost)
every $z\in P$, denote by $n_l(z)$ the $l$-th return time
of $z$. Let 
$$P_n^l(x)=\{z\in P(x)|\ n_l(z)\ge n\},\;
N_l(x)=\max\{n|\sigma_x(P_n^l(x))> \sigma_0\}$$
Then
$$\int_B N_l(x)d\nu< \frac{l}{\sigma_0}<\infty$$
\end{thm}

\begin{remark}
$N_l(x)$ is the longest return time for the $l$-th returns of
the points in a subset of conditional measure no less than $\sigma_x(P(x))-\sigma_0$
in $P(x)$. Or equivalently, $N_l(x)$ is the smallest number such that the
set of points in $P(x)$ with $l$-th return times greater than $N_l(x)$ has
conditional measure at most $\sigma_0$.
\end{remark}

\begin{proof}
Since $\mu$ is $F$-invariant and $\mu(P)>0$, we have:
$$0<\int_{P} n_1(z)d\mu=\mu(\bigcup_{j=0}^\infty F^j(P))\le1$$
Let $F_P$ be the first return map on $P$, which preserves $\mu$, then for
each $k$,
$$\int_Pn_{k+1}(z)d\mu=\int_P(n_k(F_P(z))+n_1(z))d\mu=\int_Pn_k(z)d\mu+\int_Pn_1(z)d\mu$$
Hence
$$0<\int_{P} n_l(z)d\mu=l\int_{P} n_1(z)d\mu\le l$$
Note
$$\int_{P} n_l(z)d\mu=\sum^\infty_{j=1}\mu(P_j)$$
where $P_j=\{z\in P| n_l(z)\ge j\}$ consists of points with $l$-th return
time no less than $j$. Note $P_j^l(x)=P_j\cap P(x)$.

For every $x\in B$, let $B^j=\{x\in B|N_l(x)\ge j\}$. 
By definition, for
every $j\le N_l(x)$, $\sigma_x(P_j^l(x))>\sigma_0$.
So $x\in B^j$ iff $\sigma_x(P_j^l(x))=\sigma_x(P_j\cap
P(x))>\sigma_0$. We have $$\mu(P_j)=\int_B\sigma_x(P_j\cap P(x))d\nu>\int_{B^j}\sigma_0d\nu=\nu(B^j)\cdot\sigma_0$$
hence
$$\int_{B}N_l(x)d\nu=\sum^\infty_{j=1}\nu(B^j)<\sum^\infty_{j=1}\frac{1}{\sigma_0}\mu(P_j)\le\frac{l}{\sigma_0}$$
\end{proof}

\begin{thm}\label{ThSet}
Let $g$ be a measure preserving transformation on a probability space $(X,\nu)$.
$g$ is invertible and has no periodic point. $B$ is a subset of $X$ with $\nu(B)=\nu_0>0$. $N:B\to\mathbb{N}$ is a measurable
function such that $\tilde g(x):=g^{N(x)}(x)\in B$ for almost every $x\in
B$. Assume 
\begin{equation}\label{integral}
\int_B
N(x)d\nu=\Sigma_B<\infty
\end{equation} Then there is a subset $B'$ of $B$ such that
the following holds:
\begin{enumerate}
\item
$\nu(B')=\nu_1>0$, $\tilde g(B')=B'$ and $g_*=\tilde g|_{B'}$ is invertible
and $\nu$-preserving.
\item
$$\int_{B'}N(x)d\nu\ge \nu(\bigcup_{j=-\infty}^\infty g^j(B))$$
\end{enumerate}

We call $B'$ the kernel of $\tilde g$.
\end{thm}

\begin{remark}
This theorem is nontrivial because the map $\tilde g$ is not necessarily
the first return, but just some return.  $\tilde g$ is just a measurable
transformation on $B$ which may be neither injective nor surjective. However,
we are able to find a subset of $B$ on which it is invertible, provided
integrability of return time.
\end{remark}

\begin{remark}
The assumption that $g$ has no periodic point is not necessary in this theorem.
Periodic orbits may be removed as a null set or we can easily find a subset
consisting
of periodic orbits on which $\tilde g$ is invertible.
\end{remark}

\begin{remark}
In particular, $B$ may coincide with $X$ in the theorem. This is an interesting corollary.
\end{remark}

\begin{proof}
With possible loss of a null set we may assume that the first return map
$g_B$ on $B$ is defined everywhere and invertible. $g_B$ has no periodic
point since $g$ has not. As $\tilde g$ is some return map and $N(x)$ is measurable,
there is a measurable
function $n(x)$ such that $\tilde g(x)=g^{N(x)}(x)=g_B^{n(x)}(x)$.

Define a partial order on $B$: $x_1\prec x_2$ iff there
is $n\geq 0$ such that $g_B^n(x_1)=x_2$, i.e. $x_2$ is an image of $x_1$ under
iterates of $g_B$ (and $g$). Since $g_B$ is invertible and has no periodic point, this partial
order is well defined.

Let $O^+(x)=\{\tilde g^k(x)|k\in\mathbb{N}\cup\{0\}\}$ be the forward $\tilde
g$-orbit of $x$. We define an equivalence relation on $B$: $x_1\sim x_2$ iff $Q(x_1,x_2):=O^+(x_1)\cap O^+(x_2)\neq\emptyset$, i.e.
there are $k_1, k_2>0$ such that $\tilde g^{k_1}(x_1)=\tilde g^{k_2}(x_2)$.
Note $x_1\in O^+(x_2)$ or $x_2\in O^+(x_1)$ implies $x_1\sim x_2$, but the converse is not true. Also note within an equivalence class the partial order
 is a total order,
since $g_B$ is invertible and $\tilde g^{k_1}(x_1)=\tilde g^{k_2}(x_2)=x_3$ implies $g_B^{n_1}(x_1)=g_B^{n_2}(x_2)$
for some $n_1$ and $n_2$. We write $x_1\precsim x_2$ if $x_1\prec x_2$ and
$x_1\sim x_2$.

Here we use some facts we showed in \cite{Sun}:

\begin{lemma}\textnormal{(\cite{Sun}, Proposition 4.1)}
For almost every $x\in B$, there is $x^*\in B$ such that for every $x'\precsim x$, $x^*\in O^+(x')$, i.e.
$$H(x)=\bigcap_{x'\precsim x} O^+(x')\ne\emptyset$$
Moreover, if $x_1\precsim x_2$ then $H(x_1)\supset H(x_2)$.
\end{lemma}

\begin{lemma}\textnormal{(\cite{Sun}, Proposition 4.2)}\label{inf}
For almost every $x\in B$, there is a point $x'$ such that $x'\precsim
x$ and $x'\ne x$. Hence for almost every $x\in B$, there are infinitely many
$x'$ such that $x'\precsim x$.
\end{lemma}

Excluding a null set and its (full) $g$-orbit (the union is still a null
set), we can assume the results in the last two lemmas
hold for every $x\in B$.

Let $\tilde B=\bigcap_{j=1}^\infty \tilde g^j(B)$. Then $\tilde B$ is a measurable
subset of $B$, consisting of elements that lie in the forward $\tilde g$-orbits of infinitely
many elements of $B$. For every $x\in B$, let
$$G(x)=\bigcup_{x'\sim x}H(x')$$
Then for every $\tilde x\in G(x)$, there is $x'$ such that
$x'\sim x$ and $\tilde x\in H(x')=\bigcap_{x''\precsim x'} O^+(x'')$. By Lemma \ref{inf},
the intersection is of infinitely many forward $\tilde g$-orbits, which implies $\tilde x\in\tilde B$. So $G(x)\subset \tilde B$ for every
$x\in B$.

For
$ x\in\tilde B$, there are infinitely many elements in $B$ such that their
forward $\tilde g$-orbits pass through $ x$. They also pass through one of
the  elements
in the pre-image $\tilde g^{-1}(x)$. But by integrability of return time
(\ref{integral}), the pre-image $\tilde g^{-1}(
x)$ consists of finite number of elements. So there must be some element
in $\tilde g^{-1}(x)$ which lies in infinitely many forward $\tilde g$-orbits. Such an element belongs to $\tilde B$, hence $\tilde
g^{-1}(x)\cap\tilde B$ is nonempty.
The function
$$N'(x)=\min\{N(\tilde x)|{\tilde g(\tilde x)=x, \tilde x\in\tilde B}\}$$
is a well-defined measurable function on $\tilde B$.

Define $g'(x)=g^{-N'( x)}( x)$. For $x\in\tilde B$ note $g'(x)\in\tilde B$ and hence $g'(\tilde
B)\subset\tilde B$. So $g'$ is a measurable transformation on $\tilde B$.
 We also note that $\tilde g(g'(x))=x$ for every $x\in\tilde B$.

Let $B'=\bigcap_{j=1}^\infty (g')^j(\tilde B)$. Then $B'$ is measurable. On one hand, by definition we have $g'(B')=B'$.
For every $x\in B'$, $g'(x)\in B'$ and $\tilde g(g'(x))=x$.
This implies $\tilde g(B')\supset B'$. On the other hand,
$$\tilde g(B')=\tilde g(\bigcap_{j=1}^\infty (g')^j(\tilde B))\subset\bigcap_{j=1}^\infty
\tilde g((g')^j(\tilde B))=\bigcap_{j=1}^\infty
(g')^{j-1}(\tilde B)=B'$$
So $\tilde g(B')=B'$.

For every $x\in B$ and  every $\tilde x\in G(x)$, $\tilde g(\tilde x)\in G(x)\subset \tilde B$. We
claim

\begin{lemma}
If $\tilde x\in G(x)$, then $g'(\tilde g(\tilde x))=\tilde x$.
\end{lemma}

\begin{proof}
Assume $g'(\tilde g(\tilde x))=x_0\ne \tilde x$. Since $\tilde x\in G(x)$,  there is $x'$ such that $\tilde x$ lies on the forward $\tilde g$-orbit of every $x''$ such that $x''\precsim x'$, while $x_0\in\tilde B$ lies on infinitely
many forward $\tilde g$-orbits of the elements in the same equivalence class
with $x'$. As $g$ is invertible, there are only finitely many elements between (in
the sense of the partial order) $x'$ and $x_0$.
So there must be some (in fact, infinitely
many) $x_1$ such that $x_1\precsim x'$ and both $\tilde x$ and $x_0$ lie
on the forward $\tilde g$-orbit of $x_1$. Assume $\tilde x=\tilde g^a(x_1)$
and $x_0=\tilde g^b(x_1)$. Then $\tilde g^{a+1}(x_1)=\tilde g^{b+1}(x_1)=\tilde
g(\tilde x)$. As $g$ has no periodic orbit and $\tilde g$ as well, we must have $a=b$ and
$\tilde x=x_0$, which is a contradiction.
\end{proof}

From the lemma we know for every $\tilde x\in G(x)$, $(g')^k(\tilde g^k(\tilde
x))=\tilde x$, hence $\tilde x\in (g')^k(\tilde B)$ for every positive integer $k$. This
yields
\begin{cor} $G(x)\subset B'$ for every $x\in B$.
\end{cor}
Furthermore, 
$$\bigcup_{j=-\infty}^\infty
g^j(B')\supset \bigcup_{j=-\infty}^\infty g^j(\bigcup_{x\in B}G(x))\supset
B$$
 In particular, $B'$ is nonempty and has positive measure.
We shall show $\tilde g$ is invertible
on $B'$.

$\tilde g|_{B'}$ is surjective since we have showed $\tilde g(B')=B'$.
 
 $\tilde g|_{B'}$ is injective. If $x\in B'$ then $x\in g'(\tilde B)$ and
 there is $\tilde x\in\tilde B$ such that $x=g'(\tilde x)$. But $\tilde x=\tilde
 g(g'(\tilde x))=\tilde g(x)\in B'$. This implies that $x=g'(\tilde g(x))$ for
 every $x\in B'$.
So if $x_1,x_2\in B'$ and $\tilde g(x_1)=\tilde g(x_2)\in B'$, then $x_1=g'(\tilde
g(x_1))=g'(\tilde g(x_2))=
 x_2$.

$g_*=\tilde g|_{B'}$ preserves $\nu$. Let $D_k=\{x\in B'|N(x)=k\}$ for $k=1,2,\cdots$.
Then $B'=\bigcup_{k=1}^\infty D_k$ and $D_i\bigcap D_j=\emptyset$. $g_*(D_i)\bigcap g_*(D_j)=\emptyset$ for $i\neq j$ since $g_*$ is invertible. $g_*|_{D_k}=g^k$ preserves $\nu$.
For any measurable subset $E\subset B'$, $E=\bigcup_{1\leq n<\infty}E_n$,
where $E_n=E\bigcap D_n\subset D_n$. We have
$$\nu(g_*(E))
=\sum_{1\leq n<\infty}\nu(g_*(E_n))=\sum_{1\leq
n<\infty}\nu(E_n)=\nu(E)$$
This completes the proof of the first part.

For the second part,
consider
\begin{equation}\label{Bpp}
B''=\bigcup_{k=1}^\infty(\bigcup_{j=0}^{k-1} g^j(D_k))=\bigcup_{j=0}^\infty
g^j(\bigcup_{k=j+1}^{\infty}D_k)\subset
\bigcup_{j=0}^\infty g^j(B')
\end{equation}

\begin{lemma}\label{ide}
$B''$ is $g$-invariant, and
$$B''=\bigcup_{j=-\infty}^\infty g^j(B)$$

\end{lemma}

\begin{proof}
First note for each $k$, $g^k(D_k)=
\tilde g(D_k)\subset
 B'\subset B''$. Then
 $$\bigcup_{k=1}^\infty g^k(D_k)=\bigcup_{k=1}^\infty\tilde g(D_k)=
\tilde g(B')=B'=\bigcup_{k=1}^\infty g^0(D_k)$$
So $$g(B'')=\bigcup_{k=1}^\infty(\bigcup_{j=0}^{k-1} g^{j+1}(D_k))=(\bigcup_{k=1}^\infty(\bigcup_{j=1}^{k-1} g^j(D_k)))\cup(\bigcup_{k=1}^\infty g^k(D_k))= B''$$
 
From (\ref{Bpp}) and $B'\subset B''$, we have 
$$\bigcup_{j=-\infty}^\infty g^j(B')\subset \bigcup_{j=-\infty}^\infty g^j(B'')= B''\subset\bigcup_{j=0}^\infty g^j(B')$$
which implies
$$B\subset \bigcup_{j=-\infty}^\infty g^j(B')=B''\subset\bigcup_{j=-\infty}^\infty g^j(B)$$
So
$$\bigcup_{j=-\infty}^\infty g^j(B)\subset B''\text{ and hence } B''=\bigcup_{j=-\infty}^\infty g^j(B)$$
\end{proof}
Lemma \ref{ide} yields
$$\int_{B'}N(x)d\nu=\sum_{k=1}^\infty k\cdot\nu(D_k)\ge\nu(B'')=\nu(\bigcup_{j=-\infty}^\infty
g^j(B))$$

\end{proof}

\begin{cor}\label{est}
Let $C$ be a subset of $B'$ of positive measure such that $g_*(C)=C$. If $g$ is ergodic, then
$$\int_{C}N(x)d\nu\ge 1$$

\end{cor}
\begin{proof}
Let $C_k=\{x\in C|N(x)=k\}$ for $k=1,2,\cdots$.
Consider $$C'=\bigcup_{k=1}^\infty(\bigcup_{j=0}^{k-1} g^j(C_k))$$
Similar argument shows $g(C')=C'$. If $g$ is ergodic, then $\nu(C')=1$. Hence
$$\int_{C}N(x)d\nu=\sum_{k=1}^\infty k\cdot\nu(C_k)\ge\nu(C')= 1$$
\end{proof}

\section{Proof of the Main Theorem}

\subsection{Regular Tube}

If on the fiber direction there is no zero Lyapunov exponents, then from Pesin theory \cite{BP} we know for almost every point $z$ there is a regular neighborhood around $z$ on the
fiber. Inside each regular neighborhood we can introduce a local chart and identify
a "rectangle" with the square $[-1,1]^2$ ($z$ with $0$) in Euclidean space
(in higher dimension this should be recognized as the product of unit balls
in dimensions corresponding to contracting and expanding directions).

Fix some small number $\gamma>0$,
we can define admissible $(s,\gamma)$-curves as the graphs $\{(\theta,\psi(\theta))|\theta\in[-1,1]\}$
and admissible $(u,\gamma)$-curves as $\{(\psi(\theta),\theta)|\theta\in [-1,1]\}$,
if $\psi:[-1,1]\to[-1,1]$ is a $C^1$ map with $|\psi'|<\gamma$.
There is some $0<h<1$ such that, if in addition $|\psi(0)|<h$ then admissible
$(s,\gamma)$-curves are mapped by $F^{-1}$, while $(u,\gamma)$-curves by $F$,
to the admissible curves of the same types, respectively.

Consider admissible
$(s,\gamma)$-rectangles defined as the sets of points
$$\{(u,v)\in[-1,-1]^2|v=\omega\psi_1(u)+(1-\omega)\psi_2(u),0\le\omega\le
1\}$$
if $\psi_1$ and $\psi_2$ are admissible $(s,\gamma)$-curves. Admissible
$(u,\gamma)$-rectangles are
defined analogously. Like admissible curves, these admissible rectangles are also mapped to admissible
rectangles of the same types by $F^{-1}$
and $F$, respectively.

Let us fix small numbers $\epsilon>0$ and $r>0$. 
\begin{pro}\label{PrRT}
There is a "Regular Tube" $P$, which is a measurable subset of $X\times
Y$ satisfying the following properties:
\begin{enumerate}
\item $\mu(P)=\mu_0>0$ .

\item Let $\pi:P\to X$ be the projection to the base and let $B=\pi(P)$.
Then $\nu(B)=\nu_0>0$.

\item Let $P(x)=P\cap(\{x\}\times Y)$. There is some number $\sigma_0>0$ such that, for every $x\in B$,
$\sigma_0<\sigma_x(P(x))<\sigma_0(1+r)$.

\item For every $x\in B$, there is a rectangle $R(x)$ on the fiber $\{x\}\times
Y$ whose diameter is less than $\epsilon/2$. $P(x)\subset R(x)\subset\mathcal{R}(z)$
where $\mathcal{R}(z)$ is the Lyapunov regular neighborhood of some point
$z=(x,y)\in X\times Y$ on the fiber $\{x\}\times Y$.

\item For every $x\in B$ and $z\in P(x)$, if for some $n>0$, $F^n(z)$ returns
to $P$, i.e. $g^n(x)\in B$ and $F^n(z)\in P(g^n(x))$, then the connected component of the intersection
$F^n(R(x))\cap R(g^n(x))$ containing
$F^n(z)$, denoted by 
$$CC(F^n(R(x))\cap R(g^n(x)), F^n(z))$$
is an admissible
$(u,\gamma)$-rectangle in $R(g^n(x))$ and
$$CC(F^{-n}(R(g^n(x)))\cap R(x), z)$$
is an admissible $(s,\gamma)$-rectangle
in $R(x)$. Moreover, for $j=0,1,...,n$, on the fiber $\{g^j(x)\}\times Y$
we have
$$\mathrm{diam}F^j(CC(F^{-n}(R(g^n(x)))\cap R(x), z))<\epsilon$$

\item Applying Theorem \ref{ThSe}, we may assume that
there is some $m_1>0$ such that for every $m>m_1$ and $x\in B$,
inside any set
of $\sigma_x$-measure at least $\sigma_0/2$ on the fiber $\{x\}\times Y$,
we can find a $(d_{m}^F,\epsilon)$-separated set with cardinality at least
$\exp m(h_\mu(F)-r)$.

\end{enumerate}
\end{pro}

\begin{proof}
This regular tube can be obtained with the following
steps.
\begin{enumerate}
\item On almost every fiber, find a regular point $z\in\{x\}\times Y$ and
its regular neighborhood $\mathcal{R}(z)$. Take $R(x)\in\mathcal{R}(z)$ with
diameter less than $\epsilon/2$.
\item  Find $P(x)\subset R(x)$ satisfying property (5). There is some $\sigma_0>0$
such that $B_0=\{x|\sigma_x(P(x))>\sigma_0\}>0$. For $x\in B_0$, shrink the
size of $P(x)$ properly such that $\sigma_0<\sigma_x(P(x))<\sigma_0(1+r)$.
\item Find $m_1$ and $B\subset B_0$ such that $P=\bigcup_{x\in B}P(x)$ also
satisfies property (6) and $\nu(B)>0$. $P$ is as required.
\end{enumerate}
\end{proof}

\subsection{Control of Return Time}

We start with a regular tube $P$. Applying Theorem \ref{ThInt}, we can
find a measurable section $q:B\to P$, $\pi\circ q=Id$ such that
$$\int_B N_1(x)d\nu\le\frac{1}{\sigma_0}$$
where $N_1(x)$ is the first return time of $q(x)$.

Denote by $\chi_P$ the characteristic function of the measurable set $P$. Consider the sets
$$\mathcal{A}_n=\{z\in X\times Y|\text{ For every $k\ge n$, }$$
$$\sum_{i=1}^k \chi_P(F^iz)<k\mu_0(1+\frac
r3)\text{ and }\sum_{i=1}^{k(1+r)} \chi_P(F^iz)>k\mu_0(1+\frac
{2r}3)\}$$
(throughout this paper, numbers like $k(1+r)$ are rounded to the nearest
integer, if needed).
Since $\mu$ is ergodic, Birkhoff Theorem tells us for almost every $z$,
$$\lim_{n\to\infty}\frac 1n\sum_{i=1}^n \chi_P(F^iz)=\mu(P)=\mu_0$$
which implies
$$\lim_{n\to\infty}\mu(\mathcal{A}_n)=1$$
Let $\mathcal{B}_n=\{x|\sigma_x(P(x)\cap\mathcal{A}_n)>\sigma_0(1-r)\}$.
Then as $n\to\infty$,
$$\nu(B\backslash\mathcal{B}_n)\to 0\;\text{ and }\;\int_{B\backslash\mathcal{B}_n}
N_1(x)d\nu\to 0$$
There is $m_0$ and a measurable subset $B_1\subset B$ with the following
properties
\begin{enumerate}
\item $\nu(B_1)>\nu_0(1-r)$. Let $P'=\pi^{-1}(B_1)\cap P$. $\mu(P')>\mu_0(1-r).$
\item For $x\in B_1$, let $P^{(m)}(x)=P(x)\cap
\mathcal{A}_{m}$. For $m>m_0$, $\sigma_x(P^{(m)}(x))>\sigma_0(1-r)$.
\item Let $B_2=B\backslash B_1$. $\int_{B_2}N_1(x)<r$.

\end{enumerate}

Now let us fix $m>\max\{m_1,m_0\}$. For convenience, denote by $\mathcal{K}=[m\mu_0(1+\frac r2)]$ the integer
part of $m\mu_0(1+\frac r2)$. For large $m$,
$$m\mu_0(1+\frac r3)<\mathcal{K}<m\mu_0(1+\frac{2r}3)$$

For every $x\in B_1$, $P^{(m)}(x)>\sigma_0(1-r)>\sigma_0/2$, by property
(6) of the regular tube, there is a $(d_{m}^F,\epsilon)$-separated set $E(x)\subset P^{(m)}(x)$ with cardinality
$$|E(x)|>\exp m(h_\mu(F)-r)$$
For $z\in E(x)\subset\mathcal{A}_m$, the $\mathcal{K}$-th
return time of $z$ to $P$ is an integer number between $m+1$ and $m(1+r)$.
So there is $V(x)\subset E(x)$ with cardinality
$$|V(x)|=[\frac 1{mr}\exp m(h_\mu(F)-r)]$$
and the $\mathcal{K}$-th return times for points in $V(x)$ are the same,
denoted by $N(x)$.
The set $\bigcup_{x\in B_1} V(x)$ can be chosen to be the union of $[\frac 1{mr}\exp m(h_\mu(F)-r)]$ measurable sections over $B_1$.

$N(x)$ is a measurable function on $B_1$. We extend $N(x)$ to a measurable function on $B$: for $x\in B_2$, let $N(x)=N_1(x)$. Consider the map $\tilde
g(x)=g^{N(x)}(x)$.
$\tilde g(x)$ is well defined on $B$ and $\tilde g(B)\subset B$.
Moreover,
$$\int_{B} N(x)d\nu\le\int_{B_1}N(x)d\nu+\int_{B_2}N_1(x)d\nu<m(1+r)\cdot\nu_0+r<\infty$$
Applying Theorem \ref{ThSet} we can find the kernel $B_3\subset B$ of positive
measure such that the $\tilde g$ restricted on $B_3$ is
invertible and preserves $\mu$. We can assume that $g_*=\tilde g|_{B_3}$ is ergodic
(with respect to the measure induced by $\nu$) by taking an ergodic component
of positive $\nu$-measure.

Let $B_4=B_3\cap B_1$ and $B_5=B_3\cap B_2$. 
Let $\mathcal G(x)$ be the
first return map on $B_4$ with respect to $g_*$. Then $\mathcal{G}$ is invertible
and preserves $\nu$.
Define measurable functions $\rho$ and $l$ on $B_4$ such that
$\mathcal{G}(x)=g^{\rho(x)}(x)=g_*^{l(x)}(x)$.
For $x\in B_4$,
$$\rho(x)=\sum_{j=0}^{l(x)-1} N(g_*^j(x))$$
and $$\int_{B_4}\rho(x)d\nu=\int_{B_3}N(x)d\nu$$
Note $B_3=B_4\cup B_5\subset B_4\cup B_2$. From Corollary \ref{est},
$$1\le \int_{B_3}N(x)d\nu\le\int_{B_4}N(x)d\nu+\int_{B_2}N(x)d\nu\le m(1+r)\cdot\nu(B_4)+r$$
So 
$$\nu(B_4)\ge\frac{1-r}{m(1+r)}$$
and the average return time
$$\frac{1}{\nu(B_4)}\int_{B_4}\rho(x)d\nu\le\frac{m(1+r)\cdot\nu(B_4)+r}{\nu(B_4)}$$
$$=m(1+r)+\frac{r}{\nu(B_4)}\le \frac{m(1+r)}{1-r}$$

\subsection{Construction of Horseshoes}

We are going to construct a skew product map with base $\mathcal G$ on $B_4$ and horseshoes
on fibers. we mostly follow the argument of A.Katok (see for example,
\cite[Section 15.6]{BP}). The novelty here is that in our case, 
for $x\in B_4$ and $z\in V(x)$, $\mathcal{F}(z)=F^{\rho(x)}(z)$ does not necessarily return
to $R(g^{\rho(x)}(x))=R(\mathcal{G}(x))$.
To live with this we consider the orbits of $g_*$ and use the admissible
rectangles around the points $q(g_*^k(x))$, $k=1,2,\dots, l(x)-1$,  to carry over the horseshoe structure until
it finally returns to $R(\mathcal{G}(x))$.

 For every $x\in B_3$ and $z\in\{x\}\times Y$,
let $F_*(z)=F^{N(x)}(z)\in\{g_*(x)\}\times Y$. Note $F_*$ is invertible on
$B_3\times Y$ since $F$ and $g_*$ are invertible. If in addition $z\in P(x)$,
then $F_*(z)^{\pm 1}\in P(g_*^{\pm 1}(x))$.

For $x\in B_4$, we note $g_*^k(x)\in B_5$ for $k=1,2,\dots, l(x)-1$
and $g_*^{l(x)}(x)\in B_4$. 
If $z\in V(x)$, then the connected component $$CC(R(x)\cap F_*^{-1}(R(g_*(x))),
z)$$
is an
admissible
$(s,\gamma)$-rectangle in $R(x)$.
As $F_*(z)=F^{N(x)}(z)$ is the $\mathcal{K}$-th return of $z$
to $P$ and $N(x)>m$, and points in $V(x)$ are $(d_{m}^F,\epsilon)$-separated,
from property (5) of the set $P$, we can
conclude that this connected component contains no other points in $V(x)$.
So there are $[\frac 1{mr}\exp m(h_\mu(F)-r)]$ such connected
components, each of which contains exactly one point in $V(x)$.

Let
$$S_0(z)=CC(R(x)\cap F_*^{-1}(R(g_*(x))),z)$$
and for $k=1,2,\dots,l(x)-1$, define by induction
$$S_{k}(z)=CC(F_*(S_{k-1}(z))\cap F_*^{-1}(R(g_*^{k+1}(x))), q(g_*^k(x)))\subset
F_*(S_{k-1}(z))$$
Then for each $k$, $S_k(z)$ is part of an admissible
$(s,\gamma)$-rectangle in $R(g_*^k(x))$ such that $F_*(S_k(z))$ is an admissible
$(u,\gamma)$-rectangle in $R(g_*^{k+1}(x))$. Moreover,\\
$F_*^{-(l(x)-1)}(S_{l(x)-1}(z))\subset S_0(z)$. So we can select for each $z\in V(x)$
a point $u(z)\in F_*^{-(l(x)-1)}(S_{l(x)-1}(z))$.
Then 
$$\mathcal{F}(z')=F_*^{l(x)}(u(z))\in F_*(S_{l(x)-1}(z))\subset R(g_*^{l(x)}(x))=R(\mathcal{G}(x))$$
and
$$CC(\mathcal{F}(R(x))\cap R(\mathcal{G}(x)),\mathcal{F}(u(z)))=F_*(S_{l(x)-1}(z))\subset\mathcal{F}(S_0(z))$$
is an admissible
$(u,\gamma)$-rectangle in $R(\mathcal{G}(x))$. Note that $\mathcal{F}$ is
invertible and $S_0(z)$ are disjoint for different $z\in V(x)$. So there
are $[\frac 1{mr}\exp m(h_\mu(F)-r)]$ such rectangles. Likewise, the
pre-images
$$\mathcal{F}^{-1}(CC(\mathcal{F}(R(x))\cap R(\mathcal{G}(x)),\mathcal{F}(u(z))))=CC(R(x)\cap
\mathcal{F}^{-1}(R(\mathcal{G}(x))),u(z))$$
 are $[\frac 1{mr}\exp m(h_\mu(F)-r)]$ disjoint admissible
$(s,\gamma)$-rectangle in $R(x)$.

Let $U(x)=\{u(z)|z\in V(x)\}$.
For $x\in B_4$, consider the set
$$\Lambda(x)=\bigcap_{n\in\mathbb{Z}}\mathcal F^{-n}(\bigcup_{z'\in
U(\mathcal
{G}^{n}(x))}CC(R(\mathcal{G}^n(x))\cap
\mathcal{F}^{-1}(R(\mathcal{G}^{n+1}(x))),z'))\subset R(x)$$
Then
$\Lambda(\mathcal{G}(x))=\mathcal{F}(\Lambda(x))$. $\Lambda(x)$ is the intersection
of infinitely many layers, each of which consists of $[\frac 1{mr}\exp m(h_\mu(F)-r)]$
disjoint connected components. If $z''\in\Lambda(x)$, then for each $n\in\mathbb{Z}$,
$\mathcal{F}^n(z'')$ belongs to exactly one of the connected components in 
$$\bigcup_{z'\in
U(\mathcal
{G}^{n}(x))}CC(R(\mathcal{G}^n(x))\cap
\mathcal{F}^{-1}(R(\mathcal{G}^{n+1}(x))),z')$$
Meanwhile, for any sequence $\{z_n\in U(\mathcal{G}^n(x))\}_{n\in\mathbb{Z}}$,
The intersection
$$\bigcap_{n\in\mathbb{Z}}\mathcal F^{-n}(CC(R(\mathcal{G}^n(x))\cap
\mathcal{F}^{-1}(R(\mathcal{G}^{n+1}(x))),z_n))$$
contains exactly one point in $\Lambda(x)$.
Therefore, $\Lambda=\bigcup_{x\in B_4}\Lambda(x)$ is invariant of $\mathcal F$ and $\mathcal F|_\Lambda=(\mathcal
G,\mathcal H)$, with the base $\mathcal G$ on $B_4$ and $\mathcal H$ on the
fibers conjugate
to the full shift on $[\frac 1{mr}\exp m(h_\mu(F)-r)]$ symbols.

\subsection{Estimate of Entropy}
As $\mathcal{G}$ is ergodic, $\Lambda$ carries many ergodic invariant measures for $\mathcal F|_\Lambda$ of the form 
$$\frac{1}{\nu(B_4)}\int_{B_4}\tau_x d\nu$$
where $\tau_x$ for each $x\in B_4$ is supported on $\Lambda(x)$ and $\tau_x\circ\mathcal{F}^{-1}=\tau_{\mathcal{G}(x)}$.
Entropies of these measures vary from 0 to the topological entropy of the full shift which equals
$$\log[\frac 1{mr}\exp m(h_\mu(F)-r)]$$
Ergodic measures of arbitrary intermediate entropies can be obtained by properly
assigning weights to different symbols for the shift.
These measures induce ergodic invariant measures of $F$. 
The average return time is
$$\frac{1}{\nu(B_4)}\int_{B_4}\int_{\Lambda(x)}\rho(x)d\tau_x d\nu=\frac{1}{\nu(B_4)}\int_{B_4}\rho(x)d\nu\le m(1+r)/(1-r)$$
So the measures we constructed has the maximum entropy no less than
$$\log[\frac 1{mr}\exp m(h_\mu(F)-r)]\cdot\frac{1-r}{m(1+r)}$$
which is arbitrarily close to $h_\mu(F)$ as $r\to 0$ and $m\to\infty$. This completes
the proof of Main Theorem.



\end{document}